\documentclass[11pt]{amsart}
\usepackage{amsmath,amssymb,amsthm,enumerate}

\usepackage[all,cmtip]{xy} 
\usepackage{mathrsfs} 

\title[On the Cantor--Bendixson rank]{On the Cantor--Bendixson rank of the Grigorchuk group and the Gupta--Sidki $3$ group}
\author{Rachel Skipper}
\author{Phillip Wesolek}
\thanks{The first named author was partially supported by a grant from the Simons Foundation (\#245855 to Marcin Mazur).}
\usepackage[usenames, dvipsnames]{color}

\usepackage{cancel}

\newtheorem{thm}{Theorem}[section]

\newtheorem{obs}[thm]{Observation}
\newtheorem{prop}[thm]{Proposition}
\newtheorem{lem}[thm]{Lemma}
\newtheorem{cor}[thm]{Corollary}

\theoremstyle{definition}
\newtheorem{defn}[thm]{Definition}
\newtheorem{quest}[thm]{Question}

\newtheorem*{ack}{Acknowledgments}

\newcommand{\Nb}{\mathbb{N}}

\newcommand{\mc}[1]{\mathcal{#1}}

\newcommand{\wt}[1]{\widetilde{#1}}

\newcommand{\acts}{\curvearrowright}

\newcommand{\Aut}{\mathrm{Aut}}
\newcommand{\rist}{\mathrm{rist}}
\newcommand{\rrist}{\mathrm{st}}
\newcommand{\st}{\mathrm{st}}
\newcommand{\cb}{\mathrm{rk_{CB}}}
\newcommand{\sub}{\mathrm{Sub}}
\newcommand{\rest}{\upharpoonright}
\newcommand{\normal}{\trianglelefteq}
\newcommand{\depth}{\mathrm{depth}}
\newcommand{\comm}{\mathrm{Comm}}

\newcommand{\grp}[1]{\langle #1 \rangle}

\begin{document}

\begin{abstract}
We study the Cantor--Bendixson rank of the space of subgroups for members of a general class of finitely generated self-replicating branch groups. In particular, we show for $G$ either the Grigorchuk group or the Gupta--Sidki $3$ group, the Cantor--Bendixson rank of $\sub(G)$ is $\omega$. For each natural number $n$, we additionally characterize the subgroups of rank $n$ and give a description of subgroups in the perfect kernel.
\end{abstract}

\maketitle

\section{Introduction}
Given a group $G$, the collection of subgroups, $\sub(G)$, admits a canonical totally disconnected compact topology, called the Chabauty topology, first introduced by Chabauty in \cite{Cha50}. We note that some authors only consider the space of normal subgroups, a restriction we do not make here.  The Chabauty space was used by Gromov \cite{Gro81} and also Grigorchuk \cite{Gri84} to study certain asymptotic properties of discrete groups.  This space has been most studied in the case where the group is a free group of finite rank, see \cite{Cha00}, \cite{CG05}, and \cite{CGP07} for instance. In this case, the Chabauty space is often instead referred to as the space of marked groups.  Nevertheless, fairly little is known about the structure of $\sub(G)$ for a general group.

A natural invariant of compact topological spaces is the Cantor--Bendixson rank, which is an ordinal. The Cantor--Bendixson rank deals with understanding isolated points in the space and also understanding what new isolated points are created when these points are removed from the space. In \cite{CGP07}, Cornulier, Guyot, and Pitsch determined the isolated points in the space of marked groups. This study was extended in \cite{Cor11} where Cornulier determined the Cantor--Bendixson rank for the classes of finite groups, simple groups, and finitely generated nilpotent groups considered as points in the space of marked groups.  He also exhibited a sequence $(G_n)$ of 2 generated virtually metabelian groups with Cantor--Bendixson rank $\omega^n$. For $G$ a countable abelian group, a thorough study of $\sub(G)$ was done in \cite{CGP10} where they classify $\sub(G)$ up to homeomorphism and determine the Cantor--Bendixson rank and also an extended version of the Cantor--Bendixson rank for this class.

One naturally wishes to further understand the connections between this topological invariant and the subgroup structure of the group under consideration.  In this vein, we here answer the following question, posed by R. Grigorchuk.

\begin{quest}[Grigorchuk] 
What is the Cantor--Bendixson rank of $\sub(\Gamma)$ for $\Gamma$ the Grigorchuk group? 
\end{quest}

\begin{thm}[see Corollary~\ref{cor:grigorchuk}]\label{thm:grigorchuk_intro}
For $G$ either the Grigorchuk group or the Gupta--Sidki $3$ group, the Cantor--Bendixson rank of $\sub(G)$ is $\omega$.
\end{thm}

The above theorem follows from an analysis of a general class of finitely generated self-replicating branch groups and subsequently showing the Grigorchuk group and the Gupta--Sidki $3$ group are members of this class. This class is given by two properties. 

\begin{defn} 
For $X^*$ a regular rooted tree, a group $G\leq \Aut(X^*)$ is said to have \textbf{well-approximated subgroups} if  $\bigcap_{n\geq 0}H\st_G(n)=H$ for any finitely generated $H\leq G$, where $\st_G(n)$ is the stabilizer of the $n$-th level of the tree.
\end{defn}

Having well-approximated subgroups is exactly the conjunction of two widely studied properties.
\begin{lem}[See Lemma~\ref{lem:well-approximated}]
Let $G$ be a finitely generated subgroup of $ \Aut(X^*)$, then $G$ has well-approximated subgroups if and only if $G$ is LERF and has the CSP.
\end{lem} 

The next property which we call the Grigorchuk--Nagnibeda alternative seems, at the moment, rather exotic, and currently, only the Grigorchuk group and the Gupta--Sidki $3$ group are known to enjoy it; see Theorem~\ref{thm:G-N_alternative}.  Stating the property requires a weakening of the classical notion of a sub-direct product.

\begin{defn}
For $G$ a group and $Y$ a finite set, we say that $H\leq G^Y$ is an \textbf{infra-direct product} if $\psi_y(H)$ is of finite index in $G$ for all $y\in Y$, where $\psi_y$ is the projection onto the $y$-th coordinate.
\end{defn}

\begin{defn} 
A group $G\leq \Aut(X^*)$ is said to obey the \textbf{Grigorchuk--Nagnibeda alternative} if for every finitely generated subgroup $H\leq G$ either
\begin{enumerate}[(a)]
\item there is $y\in X^*$ such that $\psi_y(H)$ is finite, or
\item there is a spanning leaf set $Y\subseteq X^*$ such that $\rrist_H(Y)$ is an infra-direct product of $G^Y$, where $\st_H(Y)$ is the pointwise stabilizer of $Y$ in $H$.
\end{enumerate}
\end{defn}

We direct the reader to Definition~\ref{defn:spanning leaf set} for the definition of a spanning leaf set. 

The Grigorchuk--Nagnibeda alternative is inspired by a characterization of finitely generated subgroups of the Grigorchuk group given by Grigorchuk and Nagnibeda in \cite{GN08}.

\begin{thm}[See Theorem~\ref{thm:perfect_ker_char}]\label{thm:perfect_ker_char_intro}
Suppose that $G\leq \Aut(X^*)$ is a finitely generated regular branch group that is just infinite, is strongly self-replicating, has well-approximated subgroups, and obeys the Grigorchuk--Nagnibeda alternative. For $H\leq G$, $H$ is in the  perfect kernel of $\sub(G)$ if and only if either
\begin{enumerate}[(1)]
\item $\psi_y(H)$ is finite for some $y\in X^*$, or
\item $H$ is not finitely generated.
\end{enumerate}
\end{thm}

See Definition~\ref{def:strongly} for the definition of strongly self-replicating.

The subgroups that do not lie in the perfect kernel have finite Cantor--Bendixson rank, and this topologically defined rank is equal to an algebraically defined rank. For a subgroup $L$, we write $\cb(L)$ for the associated Cantor--Bendixson rank of $L$ as a point in $\sub(G)$.

\begin{defn} 
For $G$ a group and $L\leq G$, the \textbf{depth} of $L$ in $G$ is the supremum of the natural numbers $n$ for which there is a series of subgroups $G=H_0>H_1>\dots>H_n=L$ such that $|H_i : H_{i+1}|=\infty$ for all $i$. Depth zero subgroups are of finite index. We denote the depth by $\depth(L)$. 
\end{defn}

\begin{thm}[See Corollary~\ref{cor:complement_of_ker}]
Suppose that $G\leq \Aut(X^*)$ is a finitely generated regular branch group that is just infinite, is strongly self-replicating, has well-approximated subgroups, and obeys the Grigorchuk--Nagnibeda alternative. If $L\leq G$ is not a member of the perfect kernel of $\sub(G)$, then
\begin{enumerate}
\item $L$ is finitely generated, 
\item $\cb(L)<\omega$, and
\item $\cb(L)=\depth(L)$.
\end{enumerate}
\end{thm}

\begin{cor}[See Corollary~\ref{cor:CB-rank_GN-alt}]
Suppose that $G\leq \Aut(X^*)$ is a finitely generated regular branch group that is just infinite, is strongly self-replicating, has well-approximated subgroups, and obeys the Grigorchuk--Nagnibeda alternative, then $\sub(G)$ has Cantor--Bendixson rank $\omega$.
\end{cor}

\begin{ack} 
We would like to thank R. Grigorchuk for suggesting the project and for his many insightful remarks. We would also like to thank T. Nagnibeda for her many helpful comments. The second named author began the project during a visit to the American Institute of Mathematics; he thanks the institute for its hospitality. 
\end{ack}

\section{Preliminaries}

\subsection{Polish spaces and the Cantor--Bendixson derivative}

\
In this section, we introduce and define the Cantor--Bendixson rank. For additional discussion, the reader is directed to \cite{Kec95}.

\begin{defn} A topological space $X$ is called a \textbf{Polish space} if the topology of $X$ is separable and can be given by a complete metric on $X$. That is to say, $X$ is the topological space left over when one takes a complete separable metric space and forgets the metric.
\end{defn}

For any topological space $X$, the \textbf{derived set} $X'$ is the set of all limit points of $X$. It is easy to see that $X'=X\setminus Y$ where $Y$ is the subset of isolated points. Thereby, $X'$ is a closed subset of $X$.

 Derived sets allow us to define the $\beta$-th Cantor--Bendixson derivative, where $\beta$ is an ordinal. The \textbf{$\beta$-th Cantor--Bendixson derivative}, denoted by $X^{\beta}$, is defined by transfinite recursion as follows:
\begin{enumerate}[$\bullet$]
\item $X^0:=X$
\item $X^{\alpha+1}:=(X^{\alpha})'$
\item for $\lambda$ a limit ordinal, $X^{\lambda}:=\bigcap_{\alpha<\lambda}X^{\alpha}$.
\end{enumerate}

\begin{defn} For a Polish space $X$, the \textbf{Cantor--Bendixson rank} is the least $\alpha$ such that $X^{\alpha}=X^{\alpha+1}$. We denote the Cantor--Bendixson rank by $\cb(X)$.
\end{defn}

Since Polish spaces are Lindel\"{o}f, it follows that for any Polish space $X$, $\cb(X)$ is a countable ordinal. The $\cb(X)$-th derivative $X^{\cb(X)}$ is a perfect topological space; that is, every point is a limit point. The set $X^{\cb(X)}$ is called the \textbf{perfect kernel} of $X$.

This definition of rank easily extends to a rank on the points in a Polish space.

\begin{defn} For a Polish space $X$ and $x\in X$, the \textbf{Cantor--Bendixson rank} of $x$ is the ordinal $\alpha$ such that $x\in X^{\alpha}\setminus X^{\alpha+1}$. If no such $\alpha$ exists, we say $x$ has infinite rank. We denote the Cantor--Bendixson rank of $x\in X$ by $\cb(x)$. 
\end{defn}

\subsection{The Chabauty space}
Given a countable set $N$, the power set $\mc{P}(N)$ admits a compact Polish topology by identifying each element $X\in \mc{P}(N)$ with the indicator function for $X$ in $\{0,1\}^{N}$. Given a countable group $G$, the powerset $\mc{P}(G)$ is then a compact Polish space. It is an easy exercise to see that 
\[
\sub(G):=\{H\in \mc{P}(G)\mid H\leq G\}
\]
is a closed subset of $\mc{P}(G)$, hence $\sub(G)$ is a compact Polish space. The space $\sub(G)$ with this topology is called the \textbf{Chabauty space} of $G$. 

The topology of the Chabauty space $\sub(G)$ has a clopen basis consisting of sets of the form
\[
O_{A,C}:=\{H\in \sub(G)\mid \forall a\in A\;a\notin H\text{ and }\forall c\in C\; c\in H\}
\]
where $A$ and $C$ range over finite subsets of $G$.

For $H\leq G$, the Cantor--Bendixson rank of the subgroup $H$ is defined to be $\cb(H)$ where $H$ is considered as an element of $\sub(G)$. 

Let us make a few easy observations, which follow from the existence of the aforementioned basis. These observations will later be used implicitly.

\begin{obs}
For $G$ a group and a finitely generated subgroup $H\leq G$ with generating set $S$, $O_{\emptyset, S}$ is a neighborhood of $H$, so if $(H_i)_{i\in \Nb}$ is a sequence converging to $H$, then $H\leq H_i$ for all $i$ sufficiently large. 
\end{obs}

\begin{obs}\label{obs:fin_index}
If $G$ is a finitely generated group, then every finite index subgroup has Cantor--Bendixson rank zero.
\end{obs}

\begin{lem}\label{lem:sub(G)_convergence}
Let $G$ be a countable group and $(H_i)_{i\in \Nb}$ with $H_i\in \sub(G)$ a sequence converging to a subgroup $H$. If $L\in \sub(G)$, then $H_i\cap L$ converges to $H\cap L$.
\end{lem}
\begin{proof}
Fix $A$ and $C$ finite subsets of $G$ such that $O_{A,C}$ is a neighborhood of $H\cap L$. It suffices to consider the case that $A=\{a\}$ and $C=\{c\}$.

For all sufficiently large $i$, we infer that $c\in H_i\cap L$, since $H_i$ converges to $H$. If $a \notin H$, then that $H_i\rightarrow H$ ensures that $a\notin H_i\cap L$ for sufficiently large $i$.  If $a\in  H\setminus (L\cap H)$, then $a$ cannot be an element of $H_i\cap L$, since else $a\in L\cap H$.  Hence, $H_i\cap L\in O_{A,C}$ for all sufficiently large $i$. We thus conclude that $H_i\cap L\rightarrow H\cap L$.
\end{proof}

\subsection{Self-similar and branch groups}

Letting $X$ be a finite set, the free monoid generated by $X$ is denoted by $X^*$. We write $X^n$ to indicate the words of length $n$ in the monoid. One may identify $X^*$ with its Cayley graph, so that $X^*$ is the rooted tree where each vertex has $|X|$ many children. It is, however, often more convenient to simply think of $X^*$ as the free monoid. 

For $x\in X^*$, the \textbf{level of $x$}, denoted by $|x|$, is the length of the word $x$ in the alphabet $X$. If a word $x$ is a prefix of a word $y$, we write $x\sqsubseteq y$. A finite subset $Y\subseteq X^*$ is a \textbf{leaf set} if $Y$ is finite and  $x\cancel{\sqsubseteq} y$ for all distinct $x$ and $y$ in $Y$. A collection of leaf sets $(Y_i)_{i\in I}$ is \textbf{independent} if $\bigcup_{f\in F}Y_f$ is a leaf set for all finite sets $F\subseteq I$. 

\begin{defn}\label{defn:spanning leaf set}
We say that $Y\subseteq X^*$ is a \textbf{spanning leaf set} if it is a leaf set and there is $N$ such that for every $x\in X^*$ with $|x|\geq N$ there is $y\in Y$ with $y\sqsubseteq x$. The least such $N$ is called the \textbf{depth} of $Y$. 
\end{defn}

One checks that every leaf set can be extended to a spanning leaf set. 

The automorphism group of $X^*$, $\Aut(X^*)$, is the set of bijections from $X^*$ to $X^*$ that preserve the prefix relation. That is, $u$ is a prefix of $v$ if and only if $g(u)$ is a prefix of $g(v)$ for any $g\in \Aut(X^*)$. Consequently, for any word $uw$ in $X^*$, $g(uw)=g(u)g_u(w)$ for some other automorphism $g_u\in \Aut(X^*)$ which depends on $u$. We call $g_u$ the \textbf{section of $g$ at $u$}; note that some authors use ``state" instead of ``section." 

For any leaf set $Y\subseteq X^*$ and $G\leq \Aut(X^*)$, we use $G^Y$ to denote the group which acts as copy of $G$ on each subtree rooted at a vertex in $Y$. Note that $G^Y$ is canonically isomorphic to $\prod_{|Y|} G$.

For any $G\leq \Aut(X^*)$, there are several subgroups of particular importance. For any $x\in X^*$, the \textbf{stabilizer of the vertex $x$},  denoted by $\st_G(x)$, is the set of elements in $G$ which fix the vertex $x$. The \textbf{rigid stabilizer of the vertex $x$}, denoted by $\rist_G(x)$, is the set of elements which fix every vertex outside of the subtree rooted at $x$. For $Y\subseteq X^*$ a leaf set, the \textbf{rigid stabilizer of $Y$} is 
\[
\rist_G(Y):=\grp{\rist_G(y)\mid y\in Y}\cong \prod_{x\in Y} \rist_G(x);
\]
Hence, $\rist_G(Y)$ is the internal direct product of the rigid stabilizers of the vertices in $Y$. We also define $\st_G(Y):=\cap_{x\in Y} \st_G(x)$. When $Y=X^n$, we call $\rist_G(X^n)$ the \textbf{rigid stabilizer of the level $n$} and denote it by $\rist_G(n)$. Similarly, we denote $\st_G(X^n)$ by $\st_G(n)$ and call it the stabilizer of the level $n$. Note that for any $Y$, $\rist_G(Y)\leq \st_G(Y)$.

The rigid stabilizer allows one to isolate an important class of groups.
\begin{defn}
A subgroup $G\leq \Aut(X^*)$ is a \textbf{branch group} if it acts transitively on every level and $\rist_G(n)$ has finite index in $G$ for all $n\in \mathbb{N}$.
\end{defn}

For $x \in X^*$, we define the \textbf{section map} $\psi_x:\Aut(X^*)\rightarrow \Aut(X^*)$ by $g \mapsto g_x$, and for a leaf set $Y$, $\psi_Y:=\prod_{x\in Y} \psi_x$. Unless the domain of $\psi_Y$ is restricted to a subgroup of $\st_{\Aut(X^*)}(Y)$, $\psi_Y$ is not a homomorphism. When the domain is restricted to a subgroup of $\st_{\Aut(X^*)}(Y)$, $\psi_Y$ can be thought of as a projection map onto the coordinates in $Y$.

\begin{defn}\label{def:strongly}
A group $G\leq \Aut(X^*)$ is called \textbf{self similar} if $g_x\in G$ for all $g\in G$ and $x\in X^*$.  A self-similar group is called \textbf{self-replicating} if $\psi_x(\st_G(x))=G$ for all $x\in X^*$.  We say that $G$ is \textbf{strongly self-replicating} if $\psi_x(\st_G(n))=G$ for all $x\in X^n$ and $n\geq 1$. 
\end{defn}

For strongly self-replicating groups, $\st_G(n)$ is a subdirect product of $G^{X^n}$. Note that strongly self-replicating is the same notion as super strongly fractal studied in detail in \cite{UA16}.

A self-similar subgroup $G\leq \text{Aut}(X^*)$ is said to be \textbf{regular branch} if it acts transitively on every level and there is a normal subgroup $K$ with finite index in $G$ such that $K^{\{x\}}\leq K$ for all $x \in X^*$ and such that $K^{X^n}$ has finite index in $G$ for all $n$. In this case, $K$ is called a \textbf{branching subgoup} for $G$. If a group $G$ is a regular branch group, then it is also a branch group as $K^{\{x\}}\leq \rist_G(x)$, and therefore, $K^{X^n}\leq \rist_G(n)$.

 For any subgroup $G\leq \Aut(X^*)$, we say $G$ has the \textbf{congruence subgroup property}, or the \textbf{CSP}, if every subgroup of finite index contains a level stabilizer. Since in a branch group $\rist_G(n)\leq \st_G(n)$ and $\rist_G(n)$ has finite index in $G$ for all $n$, a branch group has the CSP if and only if every subgroup of finite index contains a rigid stabilizer and every rigid stabilizer contains a level stabilizer. Many of the most studied branch groups have the congruence subgroup property including the Grigorchuk group \cite{Gr00} and the Gupta--Sidki groups \cite{bgs03}, \cite{Ga16}; these will be discussed in more detail later.

\subsection{Generalities on groups}
A subgroup $L$ of a group $G$ is \textbf{separable} if it is the intersection of finite index subgroups. We say $L$ is separable in $G$ when we wish to emphasize the ambient group. We say that $G$ is \textbf{LERF} if every finitely generated subgroup is separable. 

And element $g\in G$ is said to \textbf{commensurate} $L$, if $|L:L\cap gLg^{-1}|$ and $|gLg^{-1}:L\cap gLg^{-1}|$ are finite. The collection of $g\in G$ that commensurate $L$ is denoted by $\comm_G(L)$. We say that $L$ is \textbf{commensurated} in $G$ if $\comm_G(L)=G$.
 
\begin{thm}[{Caprace--Kropholler--Reid--Wesolek, \cite[Main Theorem]{CKRW17}}]\label{thm:separable1}
	Let $G$ be a group and $L\leq G$ be a separable subgroup. If $G$ is generated by finitely many cosets of $L$ and $L$ is commensurated in $G$, then $L$ contains a finite index  subgroup which is normal and separable in $G$.	
\end{thm}

We will also need a new notion for groups acting on trees.

\begin{defn} 
A group $G\leq \Aut(X^*)$ is said to have \textbf{well-approximated subgroups} if  $\bigcap_{n\geq 0}H\st_G(n)=H$ for any finitely generated $H\leq G$.
\end{defn}

\begin{lem}\label{lem:well-approximated}
Let $G$ be a finitely generated subgroup of $ \Aut(X^*)$, then $G$ has well-approximated subgroups if and only if $G$ is LERF and has the CSP.
\end{lem}

\begin{proof}
Suppose that $G$ is LERF and $G$ has the CSP. Fixing a finitely generated $H\leq G$, we may find an $\subseteq$-decreasing sequence of finite index subgroups $O_i$ such that $\bigcap_{i\in \Nb}O_i=H$ since $G$ is LERF. On the other hand, $G$ has the CSP, so there is $n_i$ such at $\st_G(n_i)\leq O_i$ for each $i$. We now see that
\[
H\leq \bigcap_{i\in \Nb}H\st_G(n_i)\leq \bigcap_{i\in \Nb}O_i=H.
\]
It follows that $H=\bigcap_{k\in \Nb}H\st_G(k)$.

Conversely, assume that for all finitely generated $H\leq G$, we have $H=\bigcap_{n\geq 0}H\st_G(n)$. Since $H\st_G(n)$ is of finite index in $G$ for any $n$, $G$ is LERF. For $O\leq G$ with finite index, $O=\bigcap_{n \geq 0}O\st_G(n)$. Since $O$ is of finite index the sequence $(O\st_G(i))_{i\in \Nb}$ is eventually constant. We conclude that $O=O\st_G(n)$ for some $n$. Hence, $\st_G(n)\leq O$, and it follows that $G$ has the CSP. 
\end{proof}

Finally, for a group $G$, the \textbf{$FC$-center of $G$} is the set of all elements in $G$ whose conjugacy class is finite. Since the size of the conjugacy class of $gh$ is bounded by the product of the sizes of the conjugacy classes of $g$ and $h$, the $FC$-center forms a group. Since an element and its conjugate lie in the same conjugacy class, the $FC$-center is also normal.

A group is \textbf{just infinite} if it is an infinite group with every proper quotient finite.

\begin{lem}\label{lem:ji_fc} Let $K$ be a finitely generated just infinite group. If $K$ is not virtually abelian, then $K$ has a trivial $FC$-center.
\end{lem}
\begin{proof}
Let $H$ be the $FC$-center for $K$. Since $K$ is just infinite, $H$ is either trivial or has finite index in $K$. 

Assume that $H$ is of finite index in $K$.  Since $K$ is finitely generated, so is $H$. Note that every element $g$ in $H$ has finitely many conjugates in $H$, so the centralizer, $C_H(g)$, has finite index in $H$. Choosing $S$ a finite generating set for $H$, the center of $H$ is 
\[Z(H)=\bigcap_{g\in S} C_H(g).\]

The center $Z(H)$ is thus a finite intersection of subgroups of finite index in $H$ and therefore has finite index in $H$. We conclude that $Z(H)$ is of finite index in $K$, hence $K$ is virtually abelian. 
\end{proof}

\section{Infra-direct products in branch groups}

Recall that for $x\in X^*$, $\psi_x$ is the section map given by $g\mapsto g_x$ and and for a leaf set $Y$, $\psi_Y:=\prod_{x\in Y} \psi_x$.

\begin{defn}
For $G$ a group and $Y$ a finite set, we say that $H\leq G^Y$ is an \textbf{infra-direct product} if $\psi_y(G)$ is of finite index in $G$ for all $y\in Y$.
\end{defn}

We begin with an easy lemma.

\begin{lem}\label{lem:finite_index_sgroup_infra}
Let $G$ be a non-virtually abelian just infinite group, $H\leq G^Y$ be an infra-direct product, and $U\subseteq Y$ be least such that $(\psi_U)\rest_H:H\rightarrow G^U$ is injective. For any finite index subgroup $L\leq H$, $U$ is least such that $(\psi_U)\rest_L:L\rightarrow G^U$ is injective.
\end{lem}
\begin{proof}
Suppose toward a contradiction there is $U'\subsetneq U$ such that $\psi_{U'}\rest_L$ is injective, and by passing to the normal core of $L$, we may assume that $L\normal H$. Since $\psi_{U'}\rest_H$ is not injective, there is $h\in H\setminus \{1\}$ such that $\psi_{U'}(h)=1$.  For each $l\in L$, $hlh^{-1}\in L$, and $\psi_{U'}(hlh^{-1})=\psi_{U'}(l)$. As $\psi_{U'}\rest_L$ is injective, we conclude that $h$ commutes with $L$. Finding $y\in Y$ such that $\psi_y(h)$ is non-trivial, the images $\psi_y(h)$ and $\psi_y(L)$ commute. This implies that $G$ has a non-trivial $FC$-center, which contradicts Lemma~\ref{lem:ji_fc}.
\end{proof}

\subsection{Lower leaf sets}

Let $G\leq \Aut(X^*)$ be a strongly self-replicating group and suppose additionally that $G$ has the CSP and is just infinite. Let $Y\subseteq X^*$ be a spanning leaf set and $H\leq G$ be such that $\rrist_H(Y)$ is an infra-direct product of $G^Y$. 

Enumerate $Y=\{y_i\}_{i=1}^n$. A  \textbf{system of lower leaf sets} $(Y_j)_{j=0}^n$ for $H$ and $Y$ is defined recursively as follows: $Y_0:=Y$. Having defined $Y_i$, we define 
\[
Y_{i+1}:=(Y_i\setminus \{y_{i+1}\}) \cup y_{i+1}Z_{i+1}
\]
where $Z_{i+1}$ is a level $X^n$ of $X^*$ such that $\st_G(Z_{i+1})\leq \psi_{y_{i+1}}(\rrist_H(Y_i))$. We may find such a $Z_{i+1}$ since $G$ has the CSP and  $\psi_{y_{i+1}}(\rrist_H(Y_i))$ is of finite index in $G$.

Let us collect a couple of observations about systems of lower leaf sets; the proofs are exercises in the definitions and so left to the reader. We note that claim (2) uses that $G$ is strongly self-replicating.
\begin{obs}\label{obs:lower_leaf_set}For $(Y_j)_{j=0}^n$ a system of lower leaf sets for $H$ and $Y=\{y_i\}_{i=1}^n$ as above, the following hold:
\begin{enumerate}
\item $y_jZ_j\subseteq Y_i$ for all $j\leq i$.
\item for each $z\in Z_i$, $\psi_{y_iz}(\st_H(Y_i))=G$.
\item If $L$ contains $H$, then $(Y_j)_{j=0}^n$ is a system of lower leaf sets for $L$ and $Y$.
\end{enumerate}
\end{obs}

The existence of lower leaf sets provides us with a key lemma, which gives insight into the structure of subgroups $L$ such that $\st_L(Y)$ is an infra-direct product of $G^Y$ for $Y$ some spanning leaf set.

\begin{lem}\label{lem:lower_leaf_set_key_lem}
For $G$, $H$, and $(Y_j)_{j=0}^n$ as above, there is a non-empty set $W\subseteq Y_n$ such that $\psi_W:\rrist_H(Y_n)\rightarrow G^W$ is injective with a finite index image.
\end{lem}

\begin{proof}
Let $U\subseteq Y_0$ be least such that $\psi_U:\rrist_H(Y_0)\rightarrow G^U$ is injective and say $U=\{y_{i_1},\dots,y_{i_m}\}$; note that for any finite index subgroup of $\rrist_H(Y_0)$, the subset $U$ is also minimal such that the map $\psi_U$ is injective by Lemma~\ref{lem:finite_index_sgroup_infra}. We inductively build a sequence of pairs $(W_j, N_j)$ with $1\leq j\leq m$ such that the following hold:
\begin{enumerate}[(i)]
\item $W_j$ is a non-empty subset of $y_{i_j}Z_{i_j}$, where $y_{i_j}Z_{i_j}$ is as in the construction of the lower leaf sets,
\item $N_j\normal \rrist_H(Y_{i_j})$,
\item  $\psi_{W_j}(N_j)$ is of finite index in $G^{W_j}$, $\psi_{W_i}(N_j)=\{1\}$ for any $i<j$, and $\psi_{y_{i_l}}(N_j)=\{1\}$ for $l>j$, and
\item Setting $\Omega_j:=\bigcup_{i\leq j}W_i\cup \{y_{i_l}\}_{l>j}$, the map $\psi_{\Omega_j}:\rrist_H(Y_{i_j})\rightarrow G^{\Omega_j}$ is injective, and $\Omega_j\subseteq Y_{i_j}$ is minimal such that $\psi_{\Omega_j}$ is injective.\end{enumerate}

For the base case, let $W_1\subseteq y_{i_1}Z_{i_1}$ be least such that $\psi_{\Omega_1}:\rrist_H(Y_{i_1})\rightarrow G^{\Omega_1}$ is injective, where $\Omega_1=W_1\cup\{y_{i_l}\}_{l>1}$. The subgroup 
$\rrist_H(Y_{i_1})$ is of finite index in $\rrist_H(Y_0)$, so $U$ must be a minimal subset of $Y_0$ such that $\psi_U$ restricted to $\rrist_H(Y_{i_1})$ is injective, by Lemma~\ref{lem:finite_index_sgroup_infra}. Hence, $W_1$ is non-empty.  Condition (i)  is clearly satisfied, and  condition (iv) follows from our choice of $U$ and  Lemma~\ref{lem:finite_index_sgroup_infra}.

For each $w\in W_1$, the map
\[
\psi_{\Omega_1\setminus\{w\}}:\rrist_H(Y_{i_1})\rightarrow G^{\Omega_1\setminus \{w\}}
\]
fails to be injective. Let $N_w$ be the kernel. In view of Observation~\ref{obs:lower_leaf_set}, $\psi_w(\rrist_H(Y_{i_1}))=G$. Therefore, $\psi_w(N_w)$ is of finite index in $G$, as $G$ is just infinite. Setting $N_1:=\grp{N_w\mid w\in W_1}$, it now follows that $\psi_{W_1}(N_1)$ is of finite index in $G^{W_1}$. Furthermore, $ \psi_{y_{i_l}}(N_1)=\{1\}$ for $l>1$. Hence, (ii) and (iii) hold.

Suppose that we have built our sequence up to $(W_j,N_j)$. By recursion, $\psi_{\Omega_j}:\rrist_H(Y_{i_j})\rightarrow G^{\Omega_j}$ is injective and $\Omega_j$ is minimal for which the projection is injective, where
\[
\Omega_{j}:=\bigcup_{i\leq j}W_{i}\cup \{y_{i_l}\}_{l>j}.
\]
 Lemma~\ref{lem:finite_index_sgroup_infra} ensures $\Omega_j$ must be a minimal subset of $Y_{i_j}$ such that $\psi_{\Omega_j}$ restricted to $\rrist_H(Y_{i_{j+1}})$ is injective. Let $W_{j+1}\subseteq y_{i_{j+1}}Z_{i_{j+1}}$ be least such that $\psi_{\Omega_{j+1}}:\rrist_H(Y_{i_{j+1}})\rightarrow G^{\Omega_{j+1}}$ is injective, where 
\[
\Omega_{j+1}:=\bigcup_{i\leq j+1}W_{i}\cup \{y_{i_l}\}_{l>j+1}.
\]
As in the base case, Lemma~\ref{lem:finite_index_sgroup_infra} ensures that $W_{j+1}$ is non-empty and $\Omega_{j+1}$ is minimal such that $\psi_{\Omega_{j+1}}$ is injective.  Conditions (i) and (iv) are thus satisfied. 

For each $w\in W_{j+1}$, the map
\[
\psi_{\Omega_{j+1}\setminus\{w\}}:\rrist_H(Y_{i_{j+1}})\rightarrow G^{\Omega_{j+1}\setminus \{w\}}
\]
fails to be injective. Letting $N_w$ be the kernel and setting $N_{j+1}:=\grp{N_w\mid w\in W_{j+1}}$, it follows as in the base case that $\psi_{W_{j+1}}(N_{j+1})$ is of finite index in $G^{W_{j+1}}$,  $\psi_{W_i}(N_{j+1})=\{1\}$ for any $i<j+1$, and $\psi_{y_{i_l}}(N_{j+1})=\{1\}$ for $l>j+1$. Hence, (ii) and (iii) hold, and our construction is complete.

Set $W:=\bigcup_{j=1}^{m}W_j$. We see that $W=\Omega_m$, so $\psi_W:\rrist_H(Y_{i_m})\rightarrow G^W$ is injective. Lemma~\ref{lem:finite_index_sgroup_infra} implies that $\psi_W:\rrist_H(Y_n)\rightarrow G^W$ is also injective. Set  $M_j:=N_j\cap \rrist_H(Y_n)$. For each $j$, $\psi_{W_l}(M_j)=\{1\}$ for $l\neq j$. Indeed, our recursive construction ensures that $\psi_{W_l}(M_j)=\{1\}$ for $l<j$. For $l>j$, $\psi_{y_{l}}(N_j)=\{1\}$. Since $\psi_{W_l}(M_j)= \psi_Z\circ  \psi_{y_l}(M_j )$ for some $Z$ such that $y_lZ=W_l$, we conclude that $\psi_{W_l}(M_j)=\{1\}$.  On the other hand, $M_j$ is a finite index subgroup of $N_j$, so $\psi_{W_j}(M_j)$ is of finite index in $G^{W_j}$. It now follows that $\psi_{W}(\grp{M_j|1\leq j\leq m})$ is of finite index in $G^W$. Hence, $\psi_W(\rrist_H(Y_n))$ is of finite index in $G^W$.
\end{proof}

\subsection{Structure results}

We now deduce several consequences of Lemma~\ref{lem:lower_leaf_set_key_lem}, which will later be used to analyze the Chabauty space.

\begin{lem}\label{lem:subdirect_fg}
Suppose that $G\leq \Aut(X^*)$ is finitely generated and just infinite, is strongly self-replicating, and has the CSP. For $Y$ a spanning leaf set and $L\leq G$, if $\st_L(Y)\leq G^Y$ is an infra-direct product, then  $L$ is finitely generated. 
\end{lem}
\begin{proof}
Letting $(Y_j)_{j=0}^n$ be a system of lower leaf sets for $L$ and $Y$, Lemma~\ref{lem:lower_leaf_set_key_lem} supplies a non-empty $W\subseteq Y_n$ such that $\psi_W:\rrist_L(Y_n)\rightarrow G^W$ is injective with a finite index image. Hence, $\rrist_L(Y_n)$ is finitely generated, and as $\rrist_L(Y_n)$ is of finite index in $L$, $L$ is finitely generated.
\end{proof}

\begin{lem}\label{lem:comm_branch}
Suppose that $G\leq \Aut(X^*)$ is finitely generated and just infinite, is strongly self-replicating, and has the CSP. For $Y$ some spanning leaf set and $L\leq G$, if $\st_L(Y)\leq G^Y$ is an infra-direct product and separable in $G$, then $|\comm_G(L):L|<\infty$.
\end{lem}
\begin{proof} 
Applying Lemma~\ref{lem:lower_leaf_set_key_lem}, we may find a spanning leaf set $Y'$ and a non-empty $W\subseteq Y'$ such that $\psi_{W}:\rrist_L(Y')\rightarrow G^{W}$ has a finite index image. Note that $\comm_G(\rrist_L(Y'))=\comm_G(L)$ since $\rrist_L(Y')$ is of finite index in $L$. Replacing $L$ with $\rrist_L(Y')$ and $Y$ with $Y'$, we may assume that there is $W\subseteq Y$ such that $\psi_W:L\rightarrow G^W$ is an injective homomorphism with a finite index image.

Put $J:=\comm_G(L)$. The stabilizer $\rrist_J(Y)$ is an infra-direct product of $G^Y$, so by Lemma~\ref{lem:subdirect_fg}, $\rrist_J(Y)$ is finitely generated.  Theorem~\ref{thm:separable1} now supplies $\tilde{L} \normal \rrist_J(Y)$ such that $\tilde{L}$ is a finite index subgroup of $L$.  Suppose toward a contradiction that $L$ is of infinite index in $\rrist_J(Y)$.  The map 
\[
(\psi_W)\rest_{\rrist_J(Y)}:\rrist_J(Y)\rightarrow G^W
\]
must be non-injective, so $(\psi_W)\rest_{\rrist_J(Y)}$ has a non-trivial kernel $I$.  The subgroup $I$ is normal in $\rrist_J(Y)$ and intersects $L$ trivially, hence $I$ and $\wt{L}$ commute. Fix some coordinate $y\in Y$ such that $\psi_y(I)$ is non-trivial. The projection $\psi_y(\wt{L})$ is then a finite index subgroup of $G$ that centralizes $\psi_y(I)$. The non-trivial group $\psi_y(I)$ is thus contained in the FC-center of $G$ which is impossible in view of Lemma~\ref{lem:ji_fc}.  We conclude that $L$ is of finite index in $\rrist_J(Y)$. Therefore, $|J:L|<\infty$, since $|J:\rrist_J(Y)|<\infty$.
\end{proof}

\begin{cor}\label{cor:fin_extensions_branch}
Suppose that $G\leq \Aut(X^*)$ is finitely generated and just infinite, is strongly self-replicating, and has the CSP. For $Y$ some spanning leaf set and $L\leq G$, if $\st_L(Y)\leq G^Y$ is an infra-direct product and separable in $G$, then there are only finitely many subgroups $H\leq G$ such that $L\leq H\leq G$ and $|H:L|<\infty$.
\end{cor}
\begin{proof} 
Every such subgroup $H$ is such that $L\leq H\leq \comm_G(L)$. Lemma~\ref{lem:comm_branch} ensures that $|\comm_G(L):L|<\infty$, so there are finitely many such $H$. 
\end{proof}

\begin{defn} 
For $L\leq G$, the \textbf{depth} of $L$ in $G$ is the supremum of the natural numbers $n$ for which there is a series of subgroups $G=H_0>H_1>\dots>H_n=L$ such that $|H_i/H_{i+1}|=\infty$ for all $i$. Depth zero subgroups are of finite index. We denote the depth by $\depth_G(L)$. When the ambient group $G$ is clear from context, we write $\depth(L)$.
\end{defn}

\begin{lem}\label{lem:finite_depth}
Suppose that $G\leq \Aut(X^*)$ is finitely generated and just infinite, is strongly self-replicating, and has the CSP. For $Y$ some spanning leaf set and $L\leq G$, if $\st_L(Y)\leq G^Y$ is an infra-direct product, then $\depth(L)<\infty$. 
\end{lem}
\begin{proof}
Let $(Y_j)_{j=0}^n$ be a system of lower leaf sets for $L$ and $Y$ and put $m=2^{|Y_n|}$, which is the size of the power set of $Y_n$. Let 
\[
G=H_0>H_1>\dots>H_m=L
\]
be a sequence of subgroups. By Observation~\ref{obs:lower_leaf_set}, $(Y_j)_{j=0}^n$ is a system of lower leaf sets for each $H_i$. For each $i$, Lemma~\ref{lem:lower_leaf_set_key_lem} supplies a non-empty $W_i\subseteq Y_n$ such that $\psi_{W_i}:\rrist_{H_i}(Y_n)\rightarrow G^{W_i}$ is injective with a finite index image. Since $m+1$ is larger than the size of the power set of $Y_n$, there are $H_i$ and $H_j$ with $i<j$ such that $W_i=W_j$. Hence, $|H_j:H_i|<\infty$. We conclude that $\depth(L)\leq m$.
\end{proof}

\section{On the Grigorchuk--Nagnibeda alternative}

Our main theorem will consider groups which satisfy the following alternative.
\begin{defn} 
A group $G\leq \Aut(X^*)$ is said to obey the \textbf{Grigorchuk--Nagnibeda alternative} if for every finitely generated subgroup $H\leq G$, either
\begin{enumerate}[(a)]
\item there is $y\in X^*$ such that $\psi_y(H)$ is finite, or
\item there is a spanning subtree $Y\subseteq X^*$ such that $\rrist_H(Y)$ is an infra-direct product of $G^Y$.
\end{enumerate}
\end{defn}

Towards establishing our main theorem, we here examine the groups described in the cases of the alternative.

\subsection{Subgroups with a finite section group: Case (a)}

For a leaf set $T\subseteq X^*$ and $n$ greater that or equal to the depth of $T$, the \textbf{shadow} of $T$ on level $n$ is
\[
S(T,n):=\{v\in X^n\mid \exists t\in T\;t\sqsubseteq v\}.
\]
By the choice of $n$, the shadow is always a non-empty leaf set.

\begin{lem}\label{lem:shadow}
Let $G\leq \Aut(X^*)$ and $H\leq G$. If $
T:=\{v\in X^k\mid \psi_v(H) \text{ is finite}\} $  is non-empty, then $S(T,n)$ is setwise invariant under the action of $H$ on $X^n$ for all $n\geq k$
\end{lem}
\begin{proof}
The claim is immediate once we establish that $T$ is setwise invariant under the action of $H$. For $v\in T$ and $h\in H$, it suffices to show that $\psi_{h(v)}(st_H(h(v)))$ is finite, since $\st_H(h(v))$ is of finite index in $H$. We see that
\[
\psi_{h(v)}(\st_H(h(v)))=(h_v)\psi_v(\st_H(v))(h_v)^{-1},
\]
and in particular, $|\psi_{h(v)}(\st_H(h(v)))|=|(h_v)\psi_v(\st_H(v))(h_v)^{-1}|<\infty$. Thus, $h(v)\in T$, and $T$ is setwise invariant under the action of $H$. 
\end{proof}

\begin{lem}\label{lem:indp_fam}
Let $G\leq \Aut(X^*)$ and $H\leq G$. If there is $v\in X^*$ such that $\psi_v(H)$ is finite, then there is a countable independent family $(Y_i)_{i\in \Nb}$ of $H$-invariant leaf sets such that $\psi_v(H)$ is finite for all $v\in Y_i$ and $i\in \Nb$. Furthermore, for any $M\in \Nb$, we may take $|x|\geq M$ for every $x\in Y_i$ and $i\in \Nb$.
\end{lem}
\begin{proof}
For $N\geq 1$, let $T:=\{v\in X^N\mid \psi_v(H) \text{ is finite}\}$. The set $T$ is non-empty for any suitably large $N$; fix such an $N$. Let $X_T$ be the subset of $X^*$ consisting of all vertices that contain a vertex of $T$ as a prefix. In other words, $X_T$ is the union of $S(T,n)$ for all $n\geq N$. Note in particular, that $X_T$ is $H$-invariant by Lemma~\ref{lem:shadow}.

The kernel $\ker(H\acts X_T)$ equals the collection of $h\in H$ such that $h\in \st_H(T)$ and $\psi_T(h)=1$. The image $\psi_T(\st_H(T))$ is finite, since $\psi_v(H)$ is finite for each $v\in T$. Thus $(\psi_T)^{-1}(1)$ is of finite index in $\st_H(T)$. We deduce that $|H:\ker(H \acts X_T)|<\infty$.

We now build a family of $H$-invariant leaf sets $(W_i)_{i\in \Nb}$ along with  natural numbers $k_i\geq 1$ such that 
\begin{enumerate}[(i)]
\item $W_i=Y_i\sqcup Z_i$ where $Y_i$ and $Z_i$ are non-empty and $H$-invariant,
\item $W_{i}\subseteq S(Z_{i-1},k_{i})$,
\end{enumerate}

For the base case, the size of any orbit of $H$ on $S(T,k)$ is bounded by $|H:\ker(H \acts X_T)|$. Fixing $k_0\geq \max\{N, |H:\ker(H \acts X_T)|+1\}$, the action of $H$ on $S(T,k_0)$ has at least two orbits.  Set $W_0=S(T,k_0)$, observe that $W_0$ is a leaf set, and fix a partition $W_0=Y_0\sqcup Z_0$ where $Y_0$ and $Z_0$ are non-empty $H$-invariant subsets. Condition (i) is satisfied, and (ii) is vacuous.

Suppose that we have built the sequence up to $n$. The shadow $S(Z_n,l)$ is $H$-invariant for all $l>k_n$. As in the base case, we may find $k_{n+1}$ large enough such that $H$ has at least two orbits on $S(Z_n,k_{n+1})$. Set $W_{n+1}=S(Z_n,k_{n+1})$, observe that $W_{n+1}$ is a leaf set, and fix a partition $W_{n+1}=Y_{n+1}\sqcup Z_{n+1}$ where $Y_{n+1}$ and $Z_{n+1}$ are non-empty $H$-invariant subsets. Conditions (i) and (ii) are clearly satisfied. 

A straightforward induction argument shows the collection $(Y_i)_{i\in \Nb}$ is the desired independent family of leaf sets. That $\psi_v(H)$ is finite for each $v\in Y_i$ follows from the fact that  such a $v$ contains an element of $T$ as a prefix. Taking $k_0>M$ at stage zero of our construction ensures that $|x|\geq M$ for all $x\in Y_i$ and $i\in \Nb$.
\end{proof}

\begin{lem}\label{lem:blockform_1^n}
Let $G\leq \Aut(X^*)$ be a self-similar regular branch group with well-approximated subgroups. If $H\leq G$ is finitely generated and there is $v\in X^*$ such that $\psi_v(H)$ is finite, then $H$ is in the perfect kernel of $\sub(G)$.
\end{lem}
\begin{proof} Suppose that $K$ is a branching subgroup of $G$ and recall that $K$ is normal in $G$.

It suffices to show that every neighborhood of $H$ in $\sub(G)$ has continuum many elements. Since $H$ is finitely generated, a neighborhood base at $H$ has the form
 \[
V_{A}:=\{I\in \sub(G)\mid H\leq I\text{ and }\forall a\in A\; a\notin I\}
\]
where $A$ ranges over finite subsets of $G\setminus H$. It therefore suffices to show that each $V_A$ has size continuum. 

Fix a finite set $A\subseteq G\setminus H$. As $G$ has well-approximated subgroups, there is a level $M$ such that $H\rist_{G}(M)\cap A=\emptyset$. Applying Lemma~\ref{lem:indp_fam}, we obtain a countable independent family $(Y_i)_{i\in \Nb}$ of $H$-invariant leaf sets such that $|v|\geq M$ for $v\in Y_i$ and $\psi_v(H)$ is finite for all $v\in Y_i$. 

For each $\alpha\in \{0,1\}^{\Nb}$, define 
\[
J_{\alpha}=\grp{K^{\{x\}}\mid x\in Y_i\text{ and }\alpha(i)=1}
\]
where $K^{\{x\}}$ is the copy of $K$ which acts only on the tree below $x$. The sequence $(Y_i)_{i\in \Nb}$ is independent, so letting $Z:=\{i\in \Nb\mid \alpha(i)=1\}$,
\[
J_{\alpha}=\bigoplus_{i\in Z}K^{Y_i}\times \bigoplus_{j\in \Nb\setminus Z} \{1\}^{Y_j}.
\]
For $y\in Y_i$ and $g\in H$, $gK^{\{y\}}g^{-1}$  acts only on the tree below $y'=g(y)\in Y_i$, since $Y_i$ is $H$-invariant. The image $\psi_{y'}(g K^{\{y\}}g^{-1})$ equals $g_yK(g_y)^{-1}$. Since $K$ is normal in $G$ and $G$ is self-similar, it follows that $g_yK(g_y)^{-1}=K$. Hence, $gK^{\{y\}}g^{-1}=K^{\{g(y)\}}$, and $H$ normalizes $J_{\alpha}$. Furthermore, $HJ_{\alpha}\leq H\rist_G(M)$, so $HJ_{\alpha}\in V_A$. 

Suppose that $\alpha$ and $\beta$ are distinct elements of $\{0,1\}^{\Nb}$ and find $i$ such that $\alpha(i)\neq \beta(i)$. Without loss of generality, we assume that $\alpha(i)=1$ while $\beta(i)=0$. Fixing $v\in Y_i$, $(hy)_{v}=h_vy_v$ for any $hy\in (HJ_{\gamma})_{(v)}$ and $\gamma\in \{0,1\}^{\Nb}$, since $y$ must fix $v$. It now follows that $\psi_v({HJ_{\alpha}})$ is infinite while $\psi_v({HJ_{\beta}})=\psi_v(H)$ is finite. Hence, $HJ_{\alpha}\neq HJ_{\beta}$. We conclude that $V_A$ contains uncountably many subgroups of $G$, and the lemma follows.
\end{proof}

\subsection{Infra-direct product subgroups: Case (b)}

We now turn our attention to the second type of subgroup described by the alternative. We begin with a proposition.

\begin{prop}\label{prop:rank_zero}
For $G\leq \Aut(X^*)$ a finitely generated group that is LERF, the finite index subgroups of $G$ are exactly the subgroups with Cantor--Bendixson rank $0$.
\end{prop}
\begin{proof}
Any finite index subgroup has rank zero since $G$ is finitely generated. Conversely, suppose that $H\in \sub(G)$ has rank $0$. The subgroup $H$ is isolated in $\sub(G)$ and must be finitely generated as otherwise $H$ can be approximated by its finitely generated subgroups. As $G$ is LERF, $H$ is the intersection of finite index subgroups. Since $H$ is isolated in $\sub(G)$, it must itself be a finite index subgroup.
\end{proof}

\begin{thm}\label{thm:sub-direct_rank}
Suppose that $G\leq \Aut(X^*)$ is finitely generated, just infinite, and strongly self-replicating, and has well-approximated subgroups. For $Y$ a spanning leaf set, if $H\leq G$ is such that $\rrist_H(Y)$ is an infra-direct product of $G^Y$, then the $\cb(H)=\depth(H)$. In particular, $\cb(H)<\omega$.
\end{thm}
\begin{proof}
The hypotheses of Lemma~\ref{lem:finite_depth} are satisfied, so  $\depth(H)<\infty$.

We now argue by induction on $\depth(H)$ for the claim. If $\depth(H)=0$, then $H$ has finite index in $G$, and Proposition~\ref{prop:rank_zero} ensures that $\cb(H)=0$. For the successor case, suppose that $\depth(H)=n+1$ and say that $A_i$ is a sequence of distinct subgroups that converges to $H$ in $\sub(G)$. We argue that all but finitely many terms of the sequence are such that $\cb(A_i)\leq n$. By Lemma~\ref{lem:subdirect_fg}, $H$ is finitely generated, so we may assume, by possibly deleting finitely many terms from the sequence, that $H\leq A_i$ for all $i$. Corollary~\ref{cor:fin_extensions_branch} tells us that only finitely many of $A_i$ can be such that $|A_i:H|<\infty$. Possibly deleting finitely many more terms, we may assume that $|A_i:H|=\infty$ for all $i$. 

For each $i$, $|A_i:H|=\infty$, so $\depth(A_i)<\depth(H)=n+1$. Applying the inductive hypothesis, $\cb(A_i)\leq n$. It now follows that there is a neighborhood of $H$ in $\sub(G)$ consisting of subgroups of rank at most $n$. Hence, $\cb(H)\leq n+1$. 

Conversely, find a sequence $G=L_0>\dots>L_{n+1}=H$ such that $|L_i:L_{i+1}|=\infty$ for each $i$. The stabilizer $\st_{L_n}(Y)$ is an infra-direct product of $G^Y$, so by Lemma~\ref{lem:subdirect_fg}, $\st_{L_n}(Y)$ is finitely generated. The group $\st_{L_n}(Y)$ is of finite index in $L_n$, so $L_n$ is finitely generated. The group $L_n$ is thus finitely generated with depth $n$, so by the inductive hypothesis, $\cb(L_n)=n$. 

Let $O_i$ be an $\subseteq$-decreasing sequence of finite index subgroups of $G$ such that $\bigcap_{i\in \Nb}O_i=H$. The sequence $O_i\cap L_n$ converges to $H$ in $\sub(H)$. The terms $O_i\cap L_n$ have depth at most $n$, since else we contradict the depth of $H$, and on the other hand, they have depth at least $n$, witnessed by the sequence $L_0>\dots >L_{n-1}>O_i\cap L_n$. Applying the inductive hypothesis, $\cb(O_i\cap L_n)=n$. The sequence $(O_i\cap L_n)_{i\in \Nb}$ is thus a sequence of rank $n$ groups converging to $H$, so $\cb(H)\geq n+1$. Hence, $\cb(H)=n+1$, and the induction is complete.
\end{proof}

\section{The structure of $\sub(G)$}

We are now prepared to give a detailed picture of $\sub(G)$ for $G$ from a certain class of well-behaved branch groups. 

\begin{thm}\label{thm:perfect_ker_char}
Suppose that $G\leq \Aut(X^*)$ is a finitely generated regular branch group that is just infinite, is strongly self-replicating, has well-approximated subgroups, and obeys the Grigorchuk--Nagnibeda alternative. For $H\leq G$, $H$ is in the  perfect kernel of $\sub(G)$ if and only if either
\begin{enumerate}[(1)]
\item $\psi_y(H)$ is finite for some $y\in X^*$, or
\item $H$ is non-finitely generated.
\end{enumerate}
\end{thm}
\begin{proof} Let $K$ be a branching subgroup for $G$. 

Let us first see that the perfect kernel contains all subgroups of the two forms stated. If $H\leq \Gamma$ has form (1), then Lemma~\ref{lem:blockform_1^n} ensures that $H$ is in the perfect kernel. 

Suppose that $H$ is non-finitely generated and let $B\leq H$ be a finitely generated subgroup. If there is a spanning leaf set $Y\subseteq X^*$ such that $\rrist_B(Y)$ is an infra-direct product in $G^Y$, then $\rrist_H(Y)$ is an infra-direct product of $G^Y$. In view of Lemma~\ref{lem:subdirect_fg}, $\rrist_H(Y)$ is finitely generated, and it follows that $H$ is finitely generated, which contradicts our assumption on $H$. It is thus the case that no finitely generated $B\leq H$ admits a spanning leaf set $Y$ such that $\rrist_B(Y)$ is an infra-direct product of $G^Y$. Since $G$ obeys the Grigorchuk--Nagnibeda alternative, each finitely generated subgroup admits $y$ such that the group of sections at $y$ is finite. Each finitely generated subgroup of $H$ is thereby an element of the perfect kernel. Noting that $H$ is the limit of its finitely generated subgroups in $\sub(G)$ and that the perfect kernel is closed, $H$ is in the perfect kernel. 

Conversely, suppose that $H$ is an element of the perfect kernel. If $H$ is non-finitely generated, (2) holds, and we are done. Let us then suppose that $H$ is finitely generated. As $G$ satisfies the Grigorchuk--Nagnibeda alternative, either $\psi_y(H)$ is finite for some $y$, or there is a spanning leaf set $Y$ such that $\rrist_H(Y)$ is an infra-direct product of $G^Y$. In the latter case, Theorem~\ref{thm:sub-direct_rank} implies that $H$ has finite rank, which is absurd. We conclude that $\psi_v(H)$ is finite for some $y$, so $(1)$ holds.  
\end{proof}

\begin{thm}\label{thm:rank_char}
Suppose that $G\leq \Aut(X^*)$ is a finitely generated regular branch group that is just infinite, is strongly self-replicating, has well-approximated subgroups, and obeys the Grigorchuk--Nagnibeda alternative. For $L\leq G$, the following are equivalent:
\begin{enumerate}
\item $L$ is not a member of the perfect kernel of $\sub(G)$.
\item There is a spanning leaf set $Y$ such that $\rrist_L(Y)$ is an infra-direct product of $G^Y$.
\item  $\cb(L)<\omega$.
\item $\depth(L)<\infty$ and $L$ is finitely generated.
\end{enumerate}
\end{thm}
\begin{proof}
(1)$\Rightarrow$(2). If (1) holds, then $L$ is finitely generated and there is no $y\in X^*$ such that $\psi_y(L)$ is finite. The Grigorchuk--Nagnibeda alternative implies that (2) holds. 

\medskip

(2)$\Rightarrow$(3) and (2)$\Rightarrow$(4). Suppose that (2) holds. From Theorem~\ref{thm:sub-direct_rank}, (3) holds, and $\depth(L)<\infty$. In view of Lemma~\ref{lem:subdirect_fg}, $L$ is also finitely generated, so (4) holds.

\medskip
That (3)$\Rightarrow$(1) holds is immediate, and it then follows  from the previous paragraph that (3)$\Rightarrow$(4).

\medskip

(4)$\Rightarrow$(1). We will prove the contrapositive of this implication. Suppose that (1) fails. In view of Theorem~\ref{thm:perfect_ker_char} either $L$ is not finitely generated or there is $y\in X^*$ such that $\psi_y(L)$ is finite. In the former case, we are done. Suppose that there is $y$ with $\psi_y(L)$ is finite. Appealing to Lemma~\ref{lem:indp_fam}, we may find an independent family $(Y_i)_{i\in \Nb}$ of $L$-invariant leaf sets such that $\psi_x(L)=\{1\}$ for all $x\in Y_i$. 

Let $K$ be the branching subgroup of $G$ and for each $m\in \Nb$, define 
\[
J_m:=\grp{K^{x}\mid x\in \bigcup_{i=0}^mY_i}=\bigoplus_{i\leq m}K^{Y_i}\times \bigoplus_{j>m} \{1\}^{Y_j}.
\]
The group $L$ normalizes $J_m$. Furthermore, $LJ_m<LJ_{m+1}$ and $|LJ_{m+1}:LJ_m|=\infty$. The sequence $(LJ_m)_{m\in \Nb}$ thus demonstrates that $\depth(H)=\infty$. Hence, (4) fails.
\end{proof}

Theorems \ref{thm:rank_char} and \ref{thm:sub-direct_rank} now give a clean description of the subgroups not in the perfect kernel.

\begin{cor}\label{cor:complement_of_ker}
Suppose that $G\leq \Aut(X^*)$ is a finitely generated regular branch group that is just infinite, is strongly self-replicating, has well-approximated subgroups, and obeys the Grigorchuk--Nagnibeda alternative. If $L\leq G$ is not a member of the perfect kernel of $\sub(G)$, then
\begin{enumerate}
\item $L$ is finitely generated, 
\item $\cb(L)<\omega$, and
\item $\cb(L)=\depth(L)$.
\end{enumerate}
\end{cor}

We can also compute exactly the Cantor-Bendixson rank of $\sub(G)$.
\begin{cor}\label{cor:CB-rank_GN-alt}
If $G\leq \Aut(X^*)$ is a finitely generated branch group that is just infinite, is strongly self-replicating, has well-approximated subgroups, and obeys the Grigorchuk--Nagnibeda alternative, then $\sub(G)$ has Cantor--Bendixson rank $\omega$.
\end{cor}
\begin{proof}
Theorem~\ref{thm:rank_char} ensures the Cantor--Bendixson rank is at most $\omega$. Conversely, given a $n\geq 1$, let $Y\subseteq X^*$ be spanning leaf set with $|Y|>n$. Let $K$ be the branching subgroup of $G$ and for each $k\in K$ let $f_k\in K^Y$ be such that $\psi_y(f_k)=k$ for all $y\in Y$. Setting $H:=\grp{f_k\mid k\in K}$, we see that $H$ is an infra-direct product of $G^Y$. The depth of $H$ in $K^Y$ is at least $|Y|-1$, so $\depth_G(H)\geq n$. Theorem \ref{thm:sub-direct_rank} ensures that $\cb(H)=\depth(H)$, so $\cb(H)\geq n$. The Cantor--Bendixson rank of $G$ is thus $\omega$.
\end{proof}

Finally, we make an observation concerning the commensurators of subgroups outside the perfect kernel.
\begin{cor}
Suppose that $G\leq \Aut(X^*)$ is a finitely generated regular branch group that is just infinite, is strongly self-replicating, has well-approximated subgroups, and obeys the Grigorchuk--Nagnibeda alternative. If $L\leq G$ is not a member of the perfect kernel of $\sub(G)$, then $|\comm_G(L):L|$ is finite.
\end{cor}
\begin{proof}
By Theorem \ref{thm:rank_char}, there is a spanning leaf set $Y$ such that $\rrist_L(Y)$ is an infra-direct product of $G^Y$. Lemma~\ref{lem:comm_branch} now implies that $|\comm_G(L):L|<\infty$.
\end{proof}

\section{Applications}

We here show that both the Grigorchuk group and the Gupta--Sidki $3$-group obey the Grigorchuk--Nagnibeda alternative. Our proof will rely on the following ``subgroup induction" theorem due to Grigorchuk--Wilson, for the Grigorchuk group, and Garrido, for the Gupta--Sidki $3$-group.

\begin{thm}[{\cite[Theorem 3]{GrWi03}, \cite[Theorem 6]{Ga16}}]\label{thm:subgroup_induction}
Let $G$ be either the Grigorchuk group or the Gupta--Sidki $3$-group. Let $\mc{X}$ be a family of subgroups of $G$ such that the following hold:
\begin{enumerate}[(i)]
\item $\{1\}\in \mc{X}$, and $G\in \mc{X}$.
\item If $H\in \mc{X}$, then $L\in \mc{X}$ for any $L\leq G$ for which $H$ is a finite index subgroup of $L$.
\item If $H$ is a finitely generated subgroup of $\st_G(1)$ and all first level sections of $H$ are in $\mc{X}$, then $H\in \mc{X}$.
\end{enumerate}
Then, all finitely generated subgroups of $G$ are elements of $\mc{X}$.
\end{thm}

The case of the Grigorchuk group in the next theorem already follows from \cite{GN08}, although their proof has not yet been published.
\begin{thm}\label{thm:G-N_alternative}
Both the Grigorchuk group and the Gupta--Sidki $3$-group obey the Grigorchuk--Nagnibeda alternative. 
\end{thm}
\begin{proof}
Let $G$ be the Gupta--Sidki $3$-group; the proof for the Grigorchuk group is similar. Let $\mc{X}$ be the collection of subgroups that satisfy the Grigorchuk--Nagnibeda alternative. That is to say, $\mc{X}$ is the collection of all subgroups $H$ such that either $|\psi_y(H)|<\infty$ for some $y\in X^*$, or there is a spanning leaf set $Y$ such that $\rrist_H(Y)$ is an infra-direct product of $G^Y$. It suffices to show $\mc{X}$ that satisfies the conditions of Theorem~\ref{thm:subgroup_induction}.

That condition (i) holds is immediate. For (ii),  suppose that $H\in \mc{X}$ and $L\leq G$ is such that $H\leq L$ with $|L:H|<\infty$. First, if there is a spanning leaf set $Y$ such that $\rrist_H(Y)$ is an infra-direct product of $G^Y$, then clearly $\rrist_L(Y)$ is an infra-direct product of $G^Y$. Hence, $L\in \mc{X}$. Suppose next that $H$ is such that $\psi_y(H)$ is finite for some $y\in X^*$. The subgroup $\psi_y(\st_L(y))$ is a finite extension of $\psi_y(\st_H(y))$ and is thus also finite. It follows that $\psi_y(L)$ is finite. We conclude that $L \in \mc{X}$.

Finally, let us argue for (iii). Suppose that $H\leq \st_G(1)$ and $H_i:=\psi_i(H)\in \mc{X}$ for all $i\in X=\{0,1,2\}$. If $\psi_v({H_i})$ is finite for some $i$, then $\psi_{iv}(H)$ is finite, so $H\in \mc{X}$. Otherwise, say that $Y_i$ is a spanning leaf set such that $\rrist_{H_i}(Y_i)$ is an infra-direct product of $G^{Y_i}$. Put $Y
:=0Y_0\cup 1Y_1\cup 2Y_2$ and consider $\rrist_H(Y)$. The group $\rrist_H(Y)$ is of finite index in $H$, and $\psi_i(\rrist_{H}(Y))$ is contained in $\rrist_{H_i}(Y_i)$. Hence, $\psi_i(\rrist_H(Y))$ is of finite index in $\rrist_{H_i}(Y)$. For each $y\in Y_i$, we conclude that $\psi_{iy}(\rrist_H(Y))=\psi_y \circ \psi_i(\rrist_H(Y))$ is of finite index in $G$. It now follows that $\rrist_H(Y)$ is an infra-direct product of $G^Y$.
\end{proof}

It is well-known that both the Grigorchuk group and the Gupta--Sidki $3$-group  are finitely generated regular branch groups that are just infinite and have well-approximated subgroups (i.e.\ have CSP and are LERF); see \cite{bgs03,Ga16,Gr00, GrWi03}. They are both strongly self-replicating by \cite[Page 673]{GrWi03} and \cite[Proposition 2.1]{Ga16} or \cite[Propositions 5.1 and 5.3]{UA16}.

\begin{cor}\label{cor:grigorchuk} For $G$ either the Grigorchuk group or the Gupta--Sidki $3$-group, the Cantor--Bendixson rank of $\sub(G)$ is $\omega$.
\end{cor}


\bibliographystyle{amsplain}

\providecommand{\bysame}{\leavevmode\hbox to3em{\hrulefill}\thinspace}
\providecommand{\MR}{\relax\ifhmode\unskip\space\fi MR }
\providecommand{\MRhref}[2]{%
  \href{http://www.ams.org/mathscinet-getitem?mr=#1}{#2}
}
\providecommand{\href}[2]{#2}

\end{document}